\documentclass{siamltex} 

\usepackage{amsmath,amsfonts,amssymb}
\usepackage{graphicx}

\newcommand{\R}{\mathbb{R}}
\newcommand{\EE}{\mathcal{E}}
\newcommand{\T}{\mathcal{T}}
\newcommand{\LS}{\mathrm{LS}}
\newcommand{\TLS}{\mathrm{TLS}}
\renewcommand{\Im}{\mathrm{Im}}
\newcommand{\F}{\mathrm{F}}   
\newcommand{\qed}{\hfill}     

\newtheorem{remark}[theorem]{Remark}

\title{A Gauss--Newton iteration for 
Total Least Squares problems
\thanks{This document is a preliminary draft. 
The final publication is available at Elsevier via 
http://dx.doi.org/10.1007/s10543-017-0678-5.}
}
\author{Dario Fasino\thanks{
Dipartimento di Scienze Matematiche, Informatiche e Fisiche,
Universit\`a di Udine, Via delle Scienze 206, 33100 Udine, Italy.
 E-mail: dario.fasino@uniud.it}
\and Antonio Fazzi\thanks{
Gran Sasso Science Institute, Viale F. Crispi 7, 67100 LÕAquila, Italy. E-mail: antonio.fazzi@gssi.it}}

\date{}
\begin{document}
\maketitle

\begin{abstract}
The Total Least Squares solution of 
an overdetermined, approximate linear equation $Ax \approx b$ 
minimizes a nonlinear function which characterizes
the backward error. We show that a 
globally convergent variant of the Gauss--Newton
iteration can be tailored to compute that solution.
At each iteration, the proposed method requires
the solution of an ordinary least squares problem where the 
matrix $A$ is perturbed by a rank-one term.

\end{abstract}




\section{Introduction}

The Total Least Squares (TLS) problem 
is a well known technique 
for solving overdetermined linear systems of equations 
$$
   Ax \approx b, \qquad A \in \R^{m \times n}, 
   \qquad b \in \R^m \qquad (m > n),
$$
in which both the matrix $A$ and the right hand side $b$ are affected by errors. 
We consider 
the following classical definition of TLS problem, see e.g.,
\cite{golubvanloan,HuffelV_book}.

\begin{definition}[TLS problem]   \label{TLSdef}
The Total Least Squares problem with data
$A \in \R^{m \times n}$ 
and $b \in \R^m$,  with $m \geq n$, is  
\begin{equation}  \label{TLS}
   \min_{E, f}{{\| (E \mid f) \|}_{\F}} , \ \ \text{subject to} \  \ b+f \in \Im(A+E) ,
\end{equation}
where $E \in \R^{m \times n}$ and $f \in \R^m$.
Given a matrix $(\bar{E} \mid \bar{f})$
that attains the minimum in \eqref{TLS},
any $x \in \R^n$ such that 
$$ 
   (A+\bar{E})x= b+\bar{f}
$$ 
is called a {\em solution of the Total Least Squares problem}
\eqref{TLS}.
\end{definition}

Here and in what follows, $\|\cdot\|_{\F}$ denotes the 
Frobenius matrix norm, while
$(E \mid f)$ denotes the $m \times (n+1)$  
matrix whose first $n$ columns are the ones of $E$, and the last column is the vector $f$.

In various applicative situations where a mathematical model
reduces to the solution of an overdetermined, 
possibly inconsistent
linear equation $Ax \approx b$, solving that equation
in the TLS sense yields a more convenient approach than the ordinary least squares approach, in which the data matrix is assumed exact and 
errors are confined to the right-hand side $b$.

In the numerical linear algebra community,
the TLS problem was firstly introduced by Golub and Van Loan 
in \cite{golub,golubvanloan}, 
motivated by an extensive statistical literature on 
``orthogonal regression'', ``errors-in-variables'',
and ``measurement error'' methods and models. 
They proposed a numerical algorithm based on the 
singular value decomposition of the 
matrix $(A\mid b)$. 
That algorithm, which requires about $2mn^2+12n^3$ arithmetic
operations \cite[\S 12.3.2]{matcomp} essentially due to SVD computations,
is still today one of the reference methods for the solution 
of general TLS problems.

Van Huffel and Vandewalle \cite{vanhuffelvandewalle} 
extended the algorithm of Golub and Van Loan 
in order to deal with a wider class of TLS problems 
also in the multiple right-hand side case, namely,
problems having non-unique solutions, and the so called 
non-generic problems which have no solution in the sense of Definition \ref{TLSdef}.
Since then,
many variants and solution methods have been 
introduced on the basic TLS problem 
because of its occurrence in many different fields 
\cite{Overview07}.
For example, Bj\"{o}rck et al.\ \cite{bjorck} proposed a method for large scale TLS problems based on Rayleigh quotient iteration; and
efficient algorithms have been introduced for solving structured TLS problems 
where the data matrix $(A\mid b)$ has a particular structure (e.g., Hankel, Toeplitz) that 
must be preserved in the solution \cite{MVHP}.
Other variants of \eqref{TLS}
have been defined in terms of generic unitarily invariant norms
\cite{uin}.
We point the reader to \cite{markovsky} 
for a recent overview of the literature on the subject.

Throughout this paper 
we denote by
$$
    C= U \Sigma V^T, \qquad \Sigma = \mathrm{Diag}(\sigma_1, \dots, \sigma_n, \sigma_{n+1}),
$$
a singular value decomposition of $C = (A\mid b)$, 
with $\sigma_1 \geq \ldots \geq \sigma_{n+1}$.
Furthermore, all vector and matrix norms are $2$-norms,
unless stated otherwise.
The following well known statement characterizes the TLS solution 
along with the conditions for its existence and uniqueness  
\cite{matcomp,Overview07,HuffelV_book}.

\begin{theorem}   \label{thm:wellposed}
Let $v_{n+1}$ be the last column of the
matrix $V$ in the SVD $C = U\Sigma V^T$.
Define the partitioning
$v_{n+1} = (\hat v^T, \gamma)^T$ with $\hat v\in\R^n$
and $\gamma\in\R$.
A solution of \eqref{TLS} exists 
and is unique if and only if $\gamma \neq 0$
and $\sigma_n \ne \sigma_{n+1}$. 
If the TLS solution exists and is unique, it is given by 
$$
     x_{\TLS} = - (1/\gamma) \hat v.
$$

\end{theorem}

Alternative characterizations of $x_{\TLS}$ also exist,
based on the SVD of $A$, see e.g., 
\cite[Thm.\ 2.7]{HuffelV_book}.
We also mention that the two conditions appearing in the 
preceding theorem, namely, 
$\gamma \neq 0$
and $\sigma_n \ne \sigma_{n+1}$, are equivalent to the 
single inequality
$\sigma'_n > \sigma_{n+1}$, where $\sigma'_n$ is the smallest
singluar value of $A$. 
The equivalence is shown in \cite[Corollary 3.4]{HuffelV_book}.
In particular, we remark that a necessary condition for existence and uniqueness of the solution is that $A$ has maximum (column) rank.

Another popular characterization of the 
the solution of the total least squares problem with data $A$ and $b$ 
is given in terms of the function $\eta(x)$,
\begin{equation}
\label{eta}
   \eta(x) = \frac{\|Ax - b\|}{\sqrt{1 + x^Tx}} .
\end{equation}
Indeed, it was shown in \cite[Sect.\ 3]{golubvanloan} that,
under well posedness hypotheses, the 
solution $x_\TLS$ can be characterized as the 
global minimum of $\eta(x)$, by means of arguments based on the
SVD of the matrix $(A\mid b)$, and $\eta(x_\TLS) = \sigma_{n+1}$.
Actually, the function $\eta(x)$ quantifies the backward error of
an arbitrary vector $x$ as approximate solution of the equation 
$Ax = b$, as shown in the forthcoming result.

\begin{lemma}   \label{lemma:RG}
For any vector $x$ 
there exist a rank-one matrix $(\bar E\mid \bar f)$ such that
$(A+\bar E)x = b+\bar f$ and $\|(\bar E\mid \bar f)\|_{\F} = \eta(x)$.
Moreover, for every matrix $(E\mid f)$ such that
$(A+E)x = b+f$ it holds $\|(E\mid f)\|_{\F} \geq \eta(x)$.
\end{lemma}

\begin{proof}
Let $r = Ax - b$  and define
$$
   \bar E = \frac{-1}{1+x^Tx}rx^T , \qquad
   \bar f = \frac{1}{1+x^Tx}r .  
$$ 
Note that
$$
   (A + \bar E) x = Ax - \frac{x^Tx}{1+x^Tx} r
   = b + \frac{1}{1+x^Tx}r = b + \bar f .
$$
Introducing the auxiliary notation $y = (x, -1)^T\in\R^{n+1}$, we have $r = (A\mid b) y$ and $y^Ty = 1 + x^Tx$, whence
$(\bar E\mid \bar f) = -ry^T /y^Ty$. 
Therefore $(E\mid f)$ has rank one, 
$$
   \|(\bar E\mid \bar f)\|_{\F} = \|(\bar E\mid\bar f)\| 
   = \frac{\|r\|}{\|y\|}
   = \frac{\|r\|}{\sqrt{1 + x^Tx}} = \eta(x),
$$
and we have the first part of the claim.
Finally, if $(A+E)x = b+f$ then
$$
   \| (E\mid f)\|_{\F} \geq  \| (E\mid f)\| \geq
   \frac{\|(E\mid f)y\|}{\|y\|} = 
   \frac{\|Ax-b\|}{\sqrt{1+x^Tx}} =\eta(x) ,
$$
and the proof is complete.
\qed
\end{proof}

Hence, under the well posedness hypotheses recalled above,
the solution $x_{\TLS}$ is characterized as the unique minimizer of the function $\eta(x)$ in \eqref{eta}. 
In this paper, 
we exploit that variational formulation of the TLS problem  
to derive an iterative method,
based on the Gauss--Newton iteration, which constructs a sequence of approximations converging to $x_{\TLS}$.

The rest of the paper is organized as follows.
In the next section 
we derive our basic algorithm and discuss 
some of its geometric and
computational properties.
In Section 3 we introduce a step size
control which guarantees convergence and 
provides precise convergence estimates, due to a rather
involved relationship between our algorithm
and an inverse power iteration with the matrix 
$C^TC$.
Next, we we present some final comments in Section 4.
The present work is completed by Appendix \ref{iteration}, which 
contains the most technical part. In fact,
our main results  
are better discussed in a rather abstract setting,
and we devote a separate place for that 
discussion, to avoid notational ambiguities.

\section{Solving TLS problems by the Gauss--Newton iteration}

As recalled before, under reasonable assumptions
the solution $x_{\TLS}$
can be characterized as the point attaining
\begin{equation}   \label{eq:min_eta}
   \min_{x\in\R^n} \eta(x) := \frac{\|Ax-b\|}{\sqrt{1+x^Tx}} .
\end{equation}
Hereafter, we show how to approximate that minimum 
by means of the Gauss--Newton method.

\subsection{The Gauss--Newton method for nonlinear least squares problems}

Let $f:\R^n \to \R^m$ be a continuously differentiable function, $m\geq n$.
Consider the unconstrained optimization problem 
\begin{equation}
\label{fgn}
   \min_{x \in \R^n} {\| f(x) \|} . 
\end{equation}
Assume that $f$ is a nonlinear  
function, and denote its Jacobian matrix by $J(x)$. 
Finding a stationary point of $\phi(x) := \| f(x) \|^2$ is equivalent to solving the equation $\nabla \phi(x)=0$.  
The Gauss--Newton algorithm \cite[\S 8.5]{ortega}
is a popular method for solving such kind of nonlinear problems which does not require computation or estimation of
the second order derivatives of $\phi$.
Rather,  
this method attempts to solve the nonlinear least squares problem 
\eqref{fgn} by means of a sequence of 
standard least squares problems obtained by the linearization
of the function $f(x)$ around the current approximation. Hence,
unlike Newton-type methods applied to the nonlinear system $\nabla\phi(x) = 0$,
the Gauss--Newton iteration does not require the Hessian matrix of $\phi(x)$
and can be implemented with just the knowledge of $f(x)$ 
and its Jacobian matrix $J(x)$, as follows.

\bigskip
\framebox[11.5cm][l]{
\begin{tabular}{ll}
\multicolumn{2}{l}{\bf Basic Gauss--Newton method}\\
{\bf Input:} & $f(x)$, $J(x)$, $x_0$;
 $\varepsilon$,  maxit (stopping criteria)   \\ 
{\bf Output:} & $\bar x$, approximate solution of \eqref{fgn} \\ 
\end{tabular}
}

\framebox[11.5cm][l]{
\begin{tabular}{l}
   Set $k := 0$, $f_0 := f(x_0)$, $J_0 := J(x_0)$ \\
    {\tt while} $\|J_k^Tf_k\| \geq \varepsilon$ 
    and $k < $ maxit \\
 \hspace{6mm} Compute $h_k := \arg\min_h \| f_k + J_kh \|$ \\
 \hspace{6mm} Set $x_{k+1} := x_k + h_k$ \\
 \hspace{6mm} Set $k := k+1$, $f_k := f(x_k)$, $J_k := J(x_k)$ \\
    {\tt end} \\
    $\bar x := x_k$
\end{tabular}
}

\bigskip
According to this procedure, the iterate
$x_{k+1}$ is obtained by replacing the 
minimization of $\| f(x_k + h)\|$ with that of the
linearized variant $\| f(x_k) + J(x_k) h\|$.
The stopping criterion exploits the 
identity 
$\nabla\phi(x) = 2J(x)^Tf(x)$, so that the smallness of the norm
of the latter could indicate nearness to a stationary point.
The resulting iteration is locally convergent to
a solution of \eqref{fgn};
if the minimum in \eqref{fgn} is positive then
the convergence rate is typically linear, 
otherwise quadratic, see e.g., \cite[\S 8.5]{ortega}.


\subsection{A basic Gauss--Newton iteration for TLS problems}

The formulation \eqref{eq:min_eta} of TLS can be 
recast as a nonlinear least squares problem in the form
\eqref{fgn}. In fact, if we set  
\begin{equation}
\label{funzeta}
   f(x)= \mu(x) (Ax-b) , \qquad
   \mu(x) = \frac{1}{\sqrt{1 + x^Tx}} , 
\end{equation}
then we have $\eta(x) = \|f(x)\|$, and the TLS solution of 
$Ax\approx b$ coincides with the minimum point of $\|f(x)\|$.
The function $f(x)$ in \eqref{funzeta} is a smooth function
whose Jacobian matrix is 
\begin{equation}
\label{jacobianaf}
   J (x)=
   \mu(x) A - \mu(x)^3 (Ax-b) x^T .
\end{equation}

A good initial point to start up the Gauss--Newton iteration 
is given by the solution of the 
standard least squares problem associated to the same data $A$ and $b$, which we denote by $x_{\LS}$. Indeed many theoretical results prove  
that $x_\LS$ and $x_\TLS$ are usually not
too far apart from each other
and the angle between them is small, see e.g., \cite[Ch.\ 6]{HuffelV_book}
and \cite{paige}.
Hereafter, we outline our adaptation of
the Gauss--Newton method to the solution of TLS problems.

\bigskip
\framebox[11.5cm][l]{
\begin{tabular}{ll}
\multicolumn{2}{l}{\bf Algorithm GN-TLS}\\
{\bf Input:} & $A,b$ (problem data); $\varepsilon$,  maxit (stopping criteria)   \\ 
{\bf Output:} & $\hat x_\TLS$, approximate solution of \eqref{TLS} \\ 
\end{tabular}
}

\framebox[11.5cm][l]{
\begin{tabular}{l}
   Set $k := 0$ \\
   Compute $x_0 := \arg\min_x \|Ax - b\|$ \\
   Compute $f_0 := f(x_0)$ and $J_0 := J(x_0)$  via \eqref{funzeta} and \eqref{jacobianaf}\\
    {\tt while} $\|J_k^Tf_k\| \geq \varepsilon$ and $k < $ maxit\\
 \hspace{6mm} Compute $h_k := \arg\min_h \| J_kh + f_k\|$ \\
 \hspace{6mm} Set $x_{k+1} := x_k + h_k$ \\
 \hspace{6mm} Set $k := k+1$\\
 \hspace{6mm} Compute $f_k := f(x_k)$ and $J_k := J(x_k)$   
 via \eqref{funzeta} and \eqref{jacobianaf}
\\
    {\tt end} \\
    $\hat x_{\TLS} := x_k$
\end{tabular}
}


\subsection{Reducing the computational cost}

The main task required at each step of the previous algorithm is the solution of a 
standard least squares problem, whose classical approach  
by means of the QR factorization 
requires a cubic cost (about $2n^2(m-\frac{n}{3})$, see 
\cite{matcomp}) in terms of arithmetic operations. 
However, the particular structure of the Jacobian matrix
\eqref{jacobianaf} allows us to reduce this cost
to a quadratic one. 
Indeed, 
apart of scaling coefficients, the matrix $J(x)$
is a rank-one modification of the data matrix $A$.
This additive structure can be exploited in the solution 
of the least squares problem $\min_h\|J_kh+f_k\|$
that yields the Gauss--Newton step at the $k$-th iteration.

Hereafter, we 
recall from \cite[\S 12.5]{matcomp} and \cite{Daniel+76}
an algorithm that 
computes the (thin) QR factorization of the matrix  $B=A+uv^T$ 
by updating a known QR factorization of $A$, also in the rectangular case.
The steps of the algorithm are the following:

\begin{itemize}
\item 
Compute  $w=Q^Tu$ so that $B=A+uv^T=Q(R+wv^T)$.

\item Compute Givens matrices $J_{m-1}, \ldots, J_1$ such that
$$
J_1 \cdots J_{m-1}w=\pm \| w \| e_1,
$$
where $J_i$ is a plane rotation on the coordinates related to the indexes $i$ and $i+1$ and $e_1$ is the first canonical vector. 
Apply the same rotations to $R$ to obtain 
$$
H=J_1 \cdots J_{m-1}R ,
$$
which is an upper Hessenberg matrix.
Hence,
$$
(J_1 \cdots J_{m-1})(R+wv^T)= H \pm \| w \| e_1 v^T = H_1,
$$
which is again an upper Hessenberg matrix.

\item 
Compute Givens matrices $G_1,\ldots,G_{n-1}$,
where $G_i$ is a plane rotation on the coordinates $i$ and $i+1$,
such that 
$$
G_{n-1} \cdots G_1 H_1 = R_1
$$
and $R_1$ is an upper triangular matrix. Finally set
$$
Q_1=Q J_{m-1}^T \cdots J_1^T G_1^T \cdots G_{n-1}^T
$$
to obtain the sought QR factorization
$$
B=A+uv^T=Q_1 R_1.
$$

\end{itemize}
By means of this procedure, 
starting from a known QR factorization of $A$,
the overall computational cost of computing the QR 
factorization of the rank-one update $B = A + uv^T$
is about $6n^2+2mn$ flops \cite{Daniel+76}. 

This procedure can be adopted in Algorithm GN-TLS
to reduce the computational cost
of the iterative part. In fact, 
the QR factorization of the data matrix $A$ can be computed once, when solving the least squares problem needed to
compute the starting value $x_0$. 
In all subsequent iterations, the
least squares problem occurring in the computation of  
$h_k$ can be attacked by updating the
QR factorization of the matrix $J_k$
by means of the aforementioned procedure.
Consequently, the computational cost
of each iteration of GN-TLS can be reduced to quadratic.


\subsection{Geometric properties}

As recalled in Section 1,
a necessary condition for existence and uniqueness of 
$x_\TLS$ is that $A$ has full column rank. Moreover,
we can safely assume that $b\notin\mathrm{Range}(A)$,
otherwise the overdetermined system $Ax\approx b$ is
consistent and the TLS problem is trivial. Hence we can restrict our attention to the case
where $C = (A\mid b)$ has full column rank.
Under this hypothesis, Algorithm GN-TLS boasts certain 
interesting geometric properties,
which are examined hereafter.

\begin{lemma}   \label{lem:ellipsoid}
Let $f(x)$ be the function in \eqref{funzeta}. If $C=(A \mid b)$
has full column rank then the image of the function $f$, $Im(f) \subset \R^m$, 
is an open subset of the ellipsoid 
$\EE = \{v\in\R^m : v^T X v =1 \}$ 
where $X=(CC^T)^+$ is the Moore--Penrose inverse of the matrix $CC^T$. 
\end{lemma}

\begin{proof}
From \eqref{funzeta} we have
\begin{equation*}
   f(x) = \mu(x) (Ax - b) = 
   \mu(x) C \begin{pmatrix} x \\-1 \end{pmatrix}.
\end{equation*}
By hypothesis, $C^+ C = I$. Hence, the $2$-norm of $C^+ f(x)$ is
$$
  \| C^+ f(x) \| = 
  \bigg\lVert \mu(x) C^+ C \begin{pmatrix} x \\-1 \end{pmatrix}
  \bigg\rVert = 
  \mu(x)  \bigg\lVert \begin{pmatrix} x \\-1 \end{pmatrix} 
  \bigg\rVert = 1,
$$
independently on $x$. Moreover, 
$$
   1 = \| C^+ f(x) \|^2 = f(x)^T (C^+)^T C^+ f(x) 
   = f(x)^T (CC^T)^+ f(x) ,
$$
due to the equation $(C^+)^T C^+ =  (CC^T)^+$, whence $f(x)\in\EE$.
On the other hand, if $v\in\mathrm{Im}(f)$ then 
the vector $y = C^+v$
must belong to the unit sphere in $\R^{n+1}$ 
and be expressed as 
$$
   y = \begin{pmatrix} x \\-1 \end{pmatrix} \bigg/
  \bigg\lVert \begin{pmatrix} x \\-1 \end{pmatrix} 
  \bigg\rVert_2
$$
for some vector $x\in\R^n$, which is possible
if and only if $y_{n+1} < 0$, 
and we have the thesis.
\qed
\end{proof}

Consequently, the sequence $\{f(x_k)\}$ generated
by Algorithm GN-TLS belongs to the ellipsoid $\EE$
introduced in the previous lemma and, if convergent,
converges toward $f(x_\TLS)$, 
which is a point on that surface closest to the origin. 
Indeed, the semiaxes of $\EE$ are oriented as the 
left singular vectors of $C$ and their lenghts
correspond to the respective singular values.

\begin{remark}   \label{rem:fyx}
For later reference, we notice that
$y_k = C^+f(x_k)$ is a unit vector on the hemisphere
$\{y \in\R^{n+1} : \|y\| = 1 , y_{n+1} < 0\}$,
and is related to $x_k$ via the equation
$$
   C^+f(x_k) = \mu(x_k) \begin{pmatrix} x_k \\-1 \end{pmatrix} .
$$
Moreover, let $\mathcal{F}:\R^{n+1}\mapsto\R^n$
be the nonlinear function
$$
   \mathcal{F}(v) = \frac{-1}{v_{n+1}}(v_1,\ldots,v_n)^T ,
   \qquad v = (v_1,\ldots,v_{n+1})^T .
$$
Then we have the inversion formula
$f_k = f(x_k) \Longleftrightarrow x_k = \mathcal{F}(C^+f_k)$.
\end{remark}

Our next result reveals that any update $f(x+h)$ can be written 
explicitly as
a linear combination of $f(x)$ and $J(x)h$.

\begin{lemma}
\label{fx+h}
For all $x$ and $h$ it holds $f(x+h) = \tau (f(x) + \theta J(x) h)$ where
\begin{equation}
\label{thetatau}
 \theta = \frac{1}{1 + \mu(x)^2 (x^T h)}, \qquad \tau = \frac{\mu(x+h)}{\mu(x)} (1 + \mu(x)^2(x^T h)).
\end{equation}
\end{lemma}

\begin{proof}
Using the expression of the Jacobian matrix in \eqref{jacobianaf}, we have
\begin{align*}
   J(x) h & = \mu(x) A h - \mu(x)^3 (x^T h) (Ax - b) \\
   & = \mu(x) A h - \mu(x)^2 (x^T h) f(x) ,
\end{align*}
whence
$$
  \mu(x) A h = J(x) h + \mu(x)^2 (x^T h) f(x).
$$
From the equality $f(x+h) = \mu(x+h) (A(x+h) - b)$,
with simple manipulations we obtain
\begin{align*}
 \frac{\mu(x)}{\mu(x+h)} f(x+h) &= \mu(x) (A(x+h) - b)  \\
                                &= \mu(x) (Ax - b) + \mu(x) A h \\
																&= f(x) + J(x) h + \mu(x)^2 (x^T h) f(x) \\
																&= (1 + \mu(x)^2 (x^T h)) f(x) + J(x) h. 
 \end{align*}
Finally,
$$
    f(x+h) = \frac{\mu(x+h)}{\mu(x)} 
    \Big((1 + \mu(x)^2 (x^T h))f(x) + J(x) h\Big)  
$$
and the claim follows.
\qed
\end{proof}

The preceding lemma allows us to acquire a geometric
view of the iterations provided by the GN-TLS method.
In fact,
both $f(x)$ and $f(x+h)$ belong to the ellipsoid $\EE$ 
given in Lemma \ref{lem:ellipsoid}. 
On the other hand, for any $h$ and $\theta$,
the point $f(x) + \theta J(x)h$
lies in the tangent space in $f(x)$ to that ellipsoid,
and is external to it.
Hence, $f(x+h)$ is the projection, or better,
the retraction of one of such points onto the ellipsoid.
Indeed, under very general hypotheses we have $|\tau| < 1$, as shown in the forthcoming lemma.

\begin{lemma}   \label{lem:tau}
Let $\tau$ be defined as in \eqref{thetatau}. Then, 
$\tau^2 \leq 1$ with equality if and only if $x$ and $h$ are parallel.
\end{lemma}

\begin{proof}
Recall that $\mu(x) = (1+ x^Tx)^{-1/2}$. Hence,
\begin{align*}
   \tau^2 & = \frac{\mu(x+h)^2}{\mu(x)^2} 
   \left( 1 + \mu(x)^2(x^T h) \right)^2 \\
          & = \frac{1 + x^T x}{1 + (x + h)^T(x + h)} 
          \left( 1 + \frac{(x^T h)^2}{(1 + x^T x)^2} 
          + 2 \frac{x^T h}{1 + x^T x} \right) \\
          & = \frac{(1 + x^T x)^2 + (x^T h)^2 + 2 (1 + x^T x)
          (x^T h)}{(1 + x^T x)(1 + (x + h)^T(x + h))} \\
          & = \frac{(1 + x^T(x + h))^2}
          {(1 + x^T x)(1 + (x + h)^T(x + h))}.
\end{align*}
Let $v = (x, 1)^T\in\R^{n+1}$ and 
$w = (x + h, 1)^T\in\R^{n+1}$. Cauchy--Schwartz inequality yields
$$
  \tau^2 = \frac{(v^T w)^2}{(v^T v)(w^T w)} \leq 1 ,
$$
and the proof is complete.
\qed
\end{proof}

\begin{remark}   \label{rem:eta}
During the iterations of the proposed algorithm 
the value of $\eta(x_k)$ is readily available from
the identity $\eta(x_k) = \|f(x_k)\|$. 
As pointed out in Lemma \ref{lemma:RG}, that number
quantifies the backward error in $x_k$, hence
the monitoring of the sequence $\{\eta(x_k)\}$
can be used to devise a reliable termination criterion,
as far as the minimization of the backward error is of interest.
In fact, numerical experiments suggest that
a few iterations may be sufficient to obtain a
relevant reduction of the backward error 
with respect to that of $x_0$.
\end{remark}


\section{An improved variant}

In this section we 
devise a variant of the basic GN-TLS method. 
The aim of this variant 
is twofold, namely,
to ensure convergence and
to increase the convergence speed of the iteration with respect to the first version. 

As shown by Lemma \ref{fx+h}, $f(x+h)$ is the retraction  
onto the ellipsoid
not of the Gauss--Newton step $f(x) + J(x)h$,
but rather that 
of a linear combination of $f(x)$ and $f(x) + J(x)h$. 
This fact may slow down the iteration,
since $\|f(x) + J(x)h\| < \|f(x) + \theta J(x)h\|$
unless $h = 0$. 
In particular, when $x^Th >0$ equation \eqref{thetatau}
gives us $0 < \theta < 1$, so that $f(x) + \theta J(x)h$
is a convex linear combination of $f(x)$ and $f(x) + J(x)h$.

In order to improve convergence, we introduce a 
step size parameter $\alpha$ and reformulate the iteration 
as $x\mapsto x + \alpha h$ where $h$ is the Gauss--Newton step. 
The step length is chosen so that
$f(x+\alpha h)$ is the retraction onto $\EE$
of $f(x) + J(x) h$, that is,
$$
   f(x+\alpha h) = \hat \tau [ f(x) + J(x)h]
$$
for some scalar $\hat \tau \neq 0$. 
We obtain the sought value from Lemma \ref{fx+h}
by the condition $\alpha = 1/\theta$,
$$
   \alpha = 1 + \mu(x)^2 x^T(\alpha h) ,
$$
whose solution is 
\begin{equation}   \label{eq:alpha}
   \alpha = \frac{1}{1 - \mu^2(x) \, x^Th}.
\end{equation}
In summary, our revised iteration is
described by the following pseudo-code:

\bigskip
\framebox[11.5cm][l]{
\begin{tabular}{ll}
\multicolumn{2}{l}{\bf Algorithm GN-TLS with ``optimal'' step length}\\
{\bf Input:}  & $A,b$ (problem data); $\varepsilon$,  maxit (stopping criteria)   \\ 
{\bf Output:} & $\hat x_\TLS$, approximate solution of \eqref{TLS} \\ 
\end{tabular}
}

\framebox[11.5cm][l]{
\begin{tabular}{l}
   Set $k := 0$ \\
   Compute $x_0 := \arg\min_x \|Ax-b\|$ \\
   Compute $f_0 := f(x_0)$ and $J_0 := J(x_0)$  
   via \eqref{funzeta} and \eqref{jacobianaf}\\
   {\tt while} $\|J_k^Tf_k\| \geq \varepsilon$ and $k <$ maxit\\
 \hspace{6mm} Compute $h_k := \arg\min_h \| J_kh + f_k\|$ \\
 \hspace{6mm} Compute $\alpha_k$ from \eqref{eq:alpha} \\
 \hspace{6mm} Set $x_{k+1} := x_k + \alpha_k h_k$ \\
 \hspace{6mm} Set $k := k+1$, $f_k := f(x_k)$, $J_k := J(x_k)$ \\
    {\tt end}  \\
    $\hat x_{\TLS} := x_k$
\end{tabular}
}

\bigskip
The forthcoming lemma 
collects specific geometric properties of this iteration.

\begin{lemma}   \label{lem:collect}
Let $\{x_k\}$ be the sequence generated by Algorithm GN-TLS with ``optimal'' step length. Then, for all $k = 0,1,\ldots$
\begin{enumerate}
\item $f(x_k) \in \mathcal{E} = 
\{ v\in\R^m : v^T(CC^T)^+v = 1\}$.
\item
For some constant $\hat\tau_k$ it holds
$f(x_{k+1}) = \hat\tau_k [f(x_k) + J(x_k) h_k]$
where $h_k$ is the Gauss--Newton step at the $k$-th iteration.
In particular, $f(x_{k+1})$ is orthogonal to $J(x_k)h_k$.

\item 
If $h_k\neq 0$ then $\eta(x_{k+1}) < \eta(x_k)$. 

\end{enumerate}
\end{lemma}

\begin{proof}
The first two claims are straightforward.
Furthermore, from Lemma \ref{lem:tau} and
the orthogonality between $f(x_{k+1})$ and $J(x_k)h_k$ we get
\begin{align*}
   \eta(x_{k+1})^2 
   & \leq \| f(x_k) + J(x_k)h_k \|^2 \\
   & = \|f(x_k)\|^2 - \| J(x_k)h \|^2 = \eta(x_k)^2 -\|J(x_k)h_k\|^2 ,
\end{align*}  
and the last claim follows.
\qed
\end{proof}

In Appendix \ref{iteration} we analyze this iteration
in a more abstract setting, for notational convenience.
We prefer to place that analysis in a separate appendix
to avoid notational ambiguities.
The final result is that 
the sequence $\{f_k\}$ generated by the foregoing algorithm
coincides with that of a power method with the matrix $(CC^T)^+$. 
This fact provides a complete understanding of the
convergence properties of the proposed algorithm.
In particular, convergence guarantee and estimates 
are inherited from this coincidence.
The main result is the following.

\begin{theorem}
Suppose that the TLS problem defined by data $A$ and $b$ 
is well posed, and 
let $\{x_k\}$ be the sequence computed by the Algorithm
GN-TLS with optimal step size. 
Then,
$$
   \| f(x_k) - f(x_\TLS) \| = 
   O((\sigma_{n+1}/\sigma_n)^{2k}) , \qquad
   |\eta(x_k) - \sigma_{n+1}| = O((\sigma_{n+1}/\sigma_n)^{4k}) .
$$
Furthermore, 
$\| x_k - x_\TLS \| = O((\sigma_{n+1}/\sigma_n)^{2k})$.
\end{theorem}

\begin{proof}
Because of Lemma \ref{lem:collect}
the sequence $\{f(x_k)\}$ is a particular case of the 
generic iteration scheme introduced in Appendix \ref{iteration}.
Hereafter we prove that the hypotheses of Theorem \ref{thm:power}
are fulfilled, so that the claim will follow
from that theorem.

Firstly note that by Theorem \ref{thm:wellposed}
existence and uniqueness of $x_\TLS$ imply the inequalities
$\sigma_n > \sigma_{n+1} > 0$.
Let $v_{n+1} = (\hat v^T,\gamma)^T$
be as in the hypotheses of Theorem \ref{thm:wellposed}.
Consider $v_{n+1}$ oriented so that $\gamma < 0$ and let 
$Cv_{n+1} = \sigma_{n+1}u_{n+1}$.
Then $\mu(x_\TLS) = -\gamma$ and
\begin{align*}
     f(x_\TLS) & = \mu(x_\TLS) C 
     \begin{pmatrix} x_\TLS \\ -1 \end{pmatrix} \\
     & = 
     \frac{\mu(x_\TLS)}{-\gamma}Cv_{n+1} = 
     \sigma_{n+1}u_{n+1} .
\end{align*}    
Moreover,
\begin{align*}
   \frac{\sigma_{n+1}}{\mu(x_\LS)}u_{n+1}^Tf(x_\LS) & = 
   v_{n+1}^T C^T C \begin{pmatrix} x_\LS \\ -1 \end{pmatrix} \\
   & = v_{n+1}^T 
   \begin{pmatrix} A^TA & A^Tb \\ b^TA & b^Tb \end{pmatrix}
   \begin{pmatrix} x_\LS \\ -1 \end{pmatrix} \\
   & =  v_{n+1}^T 
   \begin{pmatrix} 0 \\ b^T(Ax_\LS - b) \end{pmatrix} 
   = \gamma \,
   b^T(Ax_\LS - b) > 0 ,  
\end{align*}
since $b^T(Ax_\LS - b) = b^T(AA^+ - I)b < 0$,
whence $u_{n+1}^Tf(x_\LS) > 0$.
Therefore, the first part of the claim is a direct consequence of 
Theorem \ref{thm:power}.

To complete the proof it suffices to show that there exists
a constant $c$ such that 
for sufficiently large $k$ we have
$\| x_k - x_\TLS \| \leq c \| f(x_k) - f(x_\TLS) \|$. 
Let $e_{n+1}$ be the last canonical vector
in $\mathbb{R}^{n+1}$.
Since $\lim_{k\to\infty} C^+f(x_k) = v_{n+1}$ we have
$$
   \lim_{k\to\infty} e_{n+1}^T C^+f(x_k) = \gamma < 0 .
$$
Therefore, for sufficiently large $k$ 
the sequence $\{C^+f(x_k)\}$ is contained into the set
$$
   \mathcal{Y} = \{ y\in\R^{n+1} : \|y\| = 1, \
   e_{n+1}^Ty \leq \gamma/2 \},
$$
which is closed and bounded. Within that set
the nonlinear function $\mathcal{F}$ introduced in 
Remark \ref{rem:fyx} is Lipschitz continuous.
Consequently, 
there exists a constant $L > 0$ such that for any
$y,y'\in\mathcal{Y}$ we have 
$\| \mathcal{F}(y) - \mathcal{F}(y')\| \leq L\|y - y'\|$. 
Finally, for sufficiently large $k$ we have
\begin{align*}
   \| x_k - x_\TLS \| & = 
   \| \mathcal{F}(C^+ f(x_k)) - \mathcal{F}(C^+ f(x_\TLS)) \| \\
   & \leq L \| C^+ ( f(x_k) - f(x_\TLS)) \| \\
   & \leq (L/\sigma_{n+1}) \| f(x_k) - f(x_\TLS) \|  ,
\end{align*}
and the proof is complete.
\qed
\end{proof}


\section{Conclusions}

We presented an iterative method for the solution of
generic TLS problems $Ax \approx b$ with single right hand side.
The iteration is based on the Gauss--Newton method
for the solution of nonlinear least squares problems,
endowed by a 
suitable starting point and step size choice that guarantee convergence.
In exact arithmetics, the method turns out to be related 
to an inverse power method
with the matrix $C^TC$.
The main task of the iterative method
consists of a sequence of ordinary least squares problems 
associated to a rank-one perturbation of the matrix $A$.
Such least squares problems can be attacked by means of
well known 
updating procedures for the QR factorization,
whose computational cost is quadratic.
Alternatively, one can consider the use of Krylov subspace methods
for least squares problems as, e.g., CGNR or QMR \cite[\S 10.4]{matcomp}, where the coefficient matrix is only involved  
in matrix-vector products;
if $A$ is sparse, the matrix-vector product 
$(A + uv^T)x$ can be implemented as $Ax + u(v^Tx)$,
thus reducing the computational core to a sparse matrix-vector product at each inner iteration.
Moreover, our method provides a measure 
of the backward error
associated to the current approximation, which is steadily decreasing during iteration.
Hence, iteration can be terminated as soon as 
a suitable reduction of that error is attained.
On the other hand, an increase of that error indicates that 
iteration is being spoiled by rounding errors.

The present work has been maily devoted to 
the construction and theoretical analysis of 
the iterative method.
Implementation details and 
numerical experiments on practical TLS problems 
will be consiedere§d in a further work.

\appendix

\section{An iteration on an ellipsoid}
\label{iteration}

The purpose of this appendix is to discuss the iteration in 
Algorithm GN-TLS with optimal step size, which is
rephrased hereafter in a more general setting.
Notations herein mirror those in the previous sections, 
with some exceptions.

Let $C\in\R^{p\times q}$ be a full column rank matrix, let 
$\mathcal{S} = \{s \in \R^q : \|s\| = 1\}$ and 
$\EE = \{y \in \R^p : y = Cs, s\in\mathcal{S} \}$.
Therefore,
$\EE$ is a differentiable manifold of $\R^p$; more precisely,
it is an ellipsoid whose (nontrivial) semiaxes 
are directed as the left singular vectors of $C$; and the 
corresponding singular values are the respective lengths.
If $p > q$ then $\mathrm{Range}(C)$ is a proper subspace of 
$\R^p$ and some semiaxes of $\EE$ vanish.

For any nonzero vector $z\in\mathrm{Range}(C)$
there exists a unique vector $y \in\EE$ such that
$z = \alpha y$ for some scalar $\alpha > 0$;
we say that $y$ is the {\em retraction} of $z$ onto $\EE$.

For any $f\in\EE$ let $\T_f$ be
the tangent space of $\EE$ in $f$. Hence, $\T_f$
is an affine $(q-1)$-dimensional subspace, and $f\in \T_f$.
If $f = Cs$ then it is not difficult to verify that 
$\T_f$ admits the following description:
$$
   \T_f = \{ f + Cw,\ s^Tw = 0\} .
$$
In fact, the map $s\mapsto Cs$ transforms 
tangent spaces of the unit sphere $\mathcal{S}$ into
 tangent spaces of $\EE$. 

Consider the following iteration:
\begin{itemize}
\item Choose $f_0 \in \EE$
\item For $k = 0,1,2,\ldots$
    \begin{itemize}
    \item Let $z_{k}$ be the minimum norm vector in $\T_{f_k}$
    \item Let $f_{k+1}$ be the retraction
of $z_k$ onto $\EE$.
    \end{itemize}
\end{itemize}
Owing to Lemma \ref{lem:collect} 
it is not difficult to recognize that the sequence
$\{f(x_k)\}$ produced by Algorithm GN-TLS with optimal step
length 
fits into the framework of the foregoing iteration.

Hereafter, we consider the behavior of the sequence
$\{f_k\}\subset\EE$
and of the auxiliary sequence
$\{s_k\}\subset\mathcal{S}$ defined by the equation $s_k = C^+ f_k$. 
We will prove
that the sequence $\{f_k\}$ is produced by a certain power method and converges to a point in $\EE$
corresponding to the smallest
(nontrivial) semiaxis, under appropriate circumstances.
In the subsequent
Theorem \ref{thm:power} we provide some convergence estimates.
To this aim,
we need the following preliminary result characterizing the solution 
of the least squares problem with a linear constraint.
 
\begin{lemma}
Let $A$ be a full column rank matrix and let $v$
be a nonzero vector. The solution $\bar x$
of the constrained least squares problem
$$
   \min_{x\, :\, v^Tx = 0}\| Ax - b\|
$$
is given by $\bar x = P x_{\mathrm{LS}}$ where
$x_{\mathrm{LS}} = A^+b$ is the solution of the 
unconstrained least squares problem and 
$$
   P = I - \frac{1}{v^T(A^TA)^{-1}v}(A^TA)^{-1}vv^T
$$
is the oblique projector onto $\langle v\rangle^\perp$
along $(A^TA)^{-1}v$.
\end{lemma} 

\begin{proof}
Simple computations using Lagrange multipliers,
see e.g., \cite[\S 12.1]{matcomp},
prove that $\bar x$ fulfills the linear equation
$$
   \begin{pmatrix} A^TA & v \\ v^T & 0 \end{pmatrix}
   \begin{pmatrix} \bar x \\ \lambda \end{pmatrix} = 
   \begin{pmatrix} A^Tb \\ 0 \end{pmatrix} ,
$$
for some scalar $\lambda$. To solve this equation,
consider the block triangular factorization
$$
   \begin{pmatrix} A^TA & v \\ v^T & 0 \end{pmatrix} =
   \begin{pmatrix} A^TA & 0 \\ v^T & 1 \end{pmatrix}
   \begin{pmatrix} I & w \\ 0 & -v^Tw \end{pmatrix} 
$$
where $w = (A^TA)^{-1}v$.
Solving the corresponding block triangular systems we get
$$
   \begin{pmatrix} A^TA & 0 \\ v^T & 1 \end{pmatrix}
   \begin{pmatrix} x_{\mathrm{LS}} \\ -v^Tx_{\mathrm{LS}} \end{pmatrix} = 
   \begin{pmatrix} A^Tb \\ 0 \end{pmatrix} ,
$$
and
$$
   \begin{pmatrix} I & w \\ 0 & -v^Tw \end{pmatrix} 
   \begin{pmatrix} \bar x \\ \lambda \end{pmatrix} = 
   \begin{pmatrix} x_{\mathrm{LS}} \\ -v^Tx_{\mathrm{LS}} 
   \end{pmatrix} ,
$$
with $\lambda =  -v^Tx_{\mathrm{LS}}/v^Tw$ 
and
$$
   \bar x = x_{\mathrm{LS}} - \lambda w = 
   x_{\mathrm{LS}} - \frac{v^Tx_{\mathrm{LS}}}{v^T(A^TA)^{-1}v}
   (A^TA)^{-1}v .
$$  
The claim follows by rearranging terms in the last
formula. 
\qed
\end{proof}

Let $s_k\in \mathcal{S}$ and let $f_k = Cs_k$ be 
the corresponding point on $\EE$. 
The minimum norm vector
in $\mathcal{T}_{f_k}$ can be expressed as 
$z_k = f_k + Cw_k$ where 
$$
   w_k = \arg\min_{w \, :\, s_k^Tw = 0}\| f_k + Cw \| .
$$   
A straightforward application of the preceding lemma
yields the formula
\begin{align*}
   w_k & = -\Big( I - \frac{1}{s_k^T(C^TC)^{-1}s_k}(C^TC)^{-1}
         s_ks_k^T \Big) s_k \\
       & = \frac{1}{s_k^T(C^TC)^{-1}s_k}(C^TC)^{-1}s_k
          - s_k .
\end{align*}
In fact, the solution of the unconstrained problem
$\min_w \| f_k + Cw \|$ clearly is $w_\LS = -s_k$,
and $s_k^Ts_k = 1$.
Then,
the minimum norm vector in $\mathcal{T}_{f_k}$ 
admits the expression 
$$
   z_k = C(s_k + w_k) = \alpha_k C (C^TC)^{-1} s_k ,
   \qquad
   \alpha_k = \frac{1}{s_k^T(C^TC)^{-1}s_k} .
$$
Since $f_{k+1}$ is the retraction of 
$z_k$ onto $\EE$ and $C^+C = I$, we conclude that $f_{k+1} = Cs_{k+1}$ with
\begin{equation}   \label{eq:sk_pow}
    s_{k+1} = C^+ f_{k+1} = \beta_k (C^TC)^{-1} s_k ,
    \qquad \beta_k = 1/\|(C^TC)^{-1} s_k\| .
\end{equation}
Finally, 
$$
   f_{k+1} = \beta_k C(C^TC)^{-1}C^+ f_k
   = \beta_k (C^+)^TC^+ f_k = \beta_k (CC^T)^+ f_k ,
$$
as $(C^+)^TC^+ = (CC^T)^+$. 
Therefore, the sequence $\{s_k\}$ 
coincides with a sequence
obtained by the normalized inverse power method for the matrix $C^TC$,
and the sequence $\{f_k\}$ 
coincides with a properly normalized sequence
obtained by the power method for the matrix $(CC^T)^+$.
We are now in position to describe the asymptotic behavior 
of $\{f_k\}$.

\begin{theorem}   \label{thm:power}
Let $\sigma_1 \geq \ldots \geq \sigma_{q-1} > \sigma_q > 0$
be the singular values of $C$, 
and let $u_q$ be an unitary left singular vector associated 
to $\sigma_q$.
If $u_q$ is oriented so that $u_q^Tf_0 > 0$ then
$$
   \| f_k - \sigma_q u_q \| = O((\sigma_q/\sigma_{q-1})^{2k})
   ,\qquad
   | \| f_k \| - \sigma_q | = O((\sigma_q/\sigma_{q-1})^{4k}) .
$$
\end{theorem}
 
\begin{proof}
As shown in equation \eqref{eq:sk_pow},
the sequence $\{s_k\}$ corresponds to a power method 
for the matrix $(C^TC)^{-1}$ with normalization.
The spectral decomposition of $(C^TC)^{-1}$
can be readily obtained from the SVD $C = U\Sigma V^T$,
$$
   (C^TC)^{-1} = V\Lambda V^T , \qquad
   \Lambda = \mathrm{diag}(\sigma_1^{-2},\ldots,\sigma_q^{-2}) .
$$
By hypotheses, the eigenvalue $\sigma_q^{-2}$ is simple and dominant, and the angle between the respective eigenvector $v_q$
and the initial vector $s_0$ is acute.
Indeed, from the identity $Cv_q = \sigma_q u_q$ we obtain
$$
   v_q^Ts_0 = \sigma_q^{-2} v_q^T C^TC s_0 =
   \sigma_q^{-1} u_q^T f_0 > 0 .
$$
For notational simplicity let $\rho = \sigma_q^2/\sigma_{q-1}^2$.
Noting that $(C^TC)^{-1}$ is symmetric and positive definite, 
classical results on convergence properties of the power method 
\cite[\S 8.2]{matcomp} give us immediately the 
asymptotic convergence estimates
$$
   \| s_k - v_q \| = O(\rho^{k}) , 
   \qquad
   s_k^T (C^TC)^{-1}s_k  - \sigma_q^{-2} 
   = O(\rho^{2k}) .
$$
The first part of the claim follows by the inequality
$$
   \| f_k - \sigma_q u_q \| = 
   \| C(s_k - v_q) \| \leq \sigma_1 \| s_k - v_q \| .
$$
Finally, using again \eqref{eq:sk_pow} we get 
\begin{align*}
   \|f_k\|^2 = s_{k}^TC^TCs_{k}
   & = \frac{s_{k-1}^T(C^TC)^{-1}C^TC(C^TC)^{-1}s_{k-1}}{\|(C^TC)^{-1}s_{k-1}\|^2} \\ 
   & = \frac{s_{k-1}^T(C^TC)^{-1}s_{k-1}}{\|(C^TC)^{-1}s_{k-1}\|^2}  
   = \frac{\sigma_q^{-2} + O(\rho^{2k})}
   {\sigma_q^{-4} + O(\rho^{2k})}  
   = \sigma_q^{2} + O(\rho^{2k}) ,
\end{align*}
and the proof is complete.
\qed
\end{proof}

\medskip
{\bf Acknowledgements.}
The first author acknowledges the support received by 
Istituto Nazionale di Alta Matematica (INdAM, Italy) 
for his research. The work of the second author has been partly 
supported
by a student research grant by University of Udine, Italy,
and performed during a visit at Vrije Universiteit Brussel, Belgium.
Both authors thank Prof. I. Markovsky for his ospitality and advice.


%

%

\end{document}